\documentclass[12pt,a4paper]{amsart}
\usepackage{mathrsfs}
\usepackage{amssymb,amsmath,mathtools,amsthm,color}
\usepackage{graphicx, cite, txfonts}
\usepackage{hyperref}
\usepackage{url}
\usepackage{setspace}
\usepackage{enumerate}

\newtheorem{theorem}{Theorem}[section]
\newtheorem{lemma}[theorem]{Lemma}
\newtheorem{facts}[theorem]{Facts}
\newtheorem{proof of lemma}[theorem]{Proof of Lemma}
\newtheorem{proposition}[theorem]{Proposition}

\newtheorem{corollary}[theorem]{Corollary}

\theoremstyle{definition}

\newtheorem{example}[theorem]{Example}

\usepackage[headheight=110pt,top=1.4in, bottom=1.4in, left=1.0in, right=1.0in]{geometry}
\hypersetup{colorlinks,linkcolor={red},citecolor={red},urlcolor={blue}}

\newtheorem{remark}[theorem]{Remark}

\numberwithin{equation}{section}
\makeatletter
\@namedef{subjclassname@2020}{\textup{2020} Mathematics Subject Classification}
\makeatother

\newcommand{\bigslant}[2]{{\raisebox{.2em}{$#1$}\left/\raisebox{-.2em}{$#2$}\right.}}
\usepackage{tikz}
\usepackage{tikz-cd}
\usetikzlibrary{matrix}
\begin{document}


\baselineskip=17pt


\title[] {A Proper Closed Subspace of the Lipschitz Dual Containing the Linear Dual}
\author[Arindam Mandal]{Arindam Mandal$^\dagger$}
\address{Arindam Mandal, School of Mathematical Sciences, National Institute of Science Education and Research
Bhubaneswar, An OCC of Homi Bhabha National Institute, P.O. Jatni, Khurda, Odisha 752050,
India.}
\email{arindam.mandal@niser.ac.in}
\thanks{The author is supported by a research fellowship from the Department of Atomic Energy (DAE), Government of India.}
\thanks{$\dagger$ \tt{Corresponding author}}
\date{}

\begin{abstract} 
Motivated by classical results of Lindenstrauss and recent developments by 
Karn and Mandal, we investigate quotient spaces of the form 
$Lip_0(X)/\mathcal{A}$, where $\mathcal{A}$ is a finite-dimensional subspace, 
showing that these quotients are dual spaces with explicitly describable preduals.

We then focus on $Lip_0^{ph}(X)$, the space of positively homogeneous 
real-valued Lipschitz functions. This space satisfies
$
X^* \subsetneq Lip_0^{ph}(X) \subsetneq Lip_0(X),
$
and is shown to be both a dual space and the preannihilator of a closed subspace 
of the Lipschitz-free space. Consequently it follows that $\bigslant{Lip_0(X)}{Lip_0^{ph}(X)}$ is also a dual space. Furthermore, with a suitable multiplication, 
$(Lip_0^{ph}(X), Lip(\cdot))$ forms a Banach algebra, 
exhibiting structural advantages over $Lip_0(X)$.
\end{abstract}


\maketitle
\pagestyle{myheadings}
\markboth{}{Duality and Algebraic Structure in Positively Homogeneous Lipschitz Spaces}
\section{Introduction}
Given a pointed metric space $(M,d,0_M)$ and a Banach space $X$, $Lip_0(M, X)$ denotes the Banach space of all Lipschitz functions from $M$ to $X$, that vanish on $0_M$ with respect to the Lipschitz norm
$$Lip(f):= \sup\limits_{x_1 \neq x_2} \frac{\|f(x_1)-f(x_2)\|}{d(x_1, x_2)}.$$
When $X = \mathbb{R}$ we simply write $Lip_0(M)$ for this space. Over the past two decades, the study of Banach spaces of Lipschitz functions and particularly their canonical preduals, the Lipschitz-free spaces has been among the most active areas of research in Banach space theory. For a comprehensive introduction and additional information, we refer the reader to \cite{ASOLFBS}, \cite{LA2}. One of the earliest results in this area is the extremely useful result by Joram Lindenstrauss~\cite[Theorem~2]{ONPIBS}, according to which 
$X^*$ is $1$-complemented closed subspace of 
$(Lip_0(X), Lip(\cdot))$ (see \cite[Chapter 7]{GNFA} for more details and applications). 
Most recently, Karn and Mandal have shown in~\cite[Theorem 2.5 and Remark~5.2]{KarnMandal2025} 
that the space of bounded linear maps $L(X, Y)$ between the Banach spaces $X$ and $Y$, is $1$-complemented in $Lip_0(X, Y)$ (whenever $Y$ is a dual space) and  the quotient space 
$\bigslant{Lip_0(X, Y)}{L(X, Y)}$ 
is linearly isometrically isomorphic to an appropriate space of operators, whenever $Y$ is a dual injective space. In particular, we observe that for any infinite-dimensional normed linear space $X$, 
the quotient of its linear dual $X^*$ within the Lipschitz dual $Lip_0(X)$ forms a well-behaved 
space of operators (indeed, a dual space). Motivated by these works, the author was led to investigate the following question:

\emph{what can be said about the quotient space $Lip_0(X)/\mathcal{A}$ when $\mathcal{A}$ is a finite-dimensional subspace of $Lip_0(X)$?}

In the present work in Section \ref{sec 2}, we prove that this quotient space is also a dual space and provide an explicit characterization of its predual. 

While searching for a proper closed subspace of $Lip_0(X)$ that contains $X^*$ properly, our attention naturally turns to $Lip_0^{ph}(X)$, the space of positively homogeneous real-valued Lipschitz functions. Interestingly, $Lip_0^{ph}(X)$ turns out to be both a dual space and the preannihilator of a closed subspace of the Lipschitz-free space, satisfying 
$$
X^* \subsetneq Lip_0^{ph}(X) \subsetneq Lip_0(X).
$$
Moreover, we define an appropriate multiplication on it (distinct from the one introduced by the authors in \cite{LAALFSOUMS}), under which  $(Lip_0^{ph}(X), Lip(\cdot))$ also becomes a Banach algebra. It is worth noting that $Lip_0(M)$, equipped with the Lipschitz norm and pointwise multiplication, fails to form a Banach algebra unless $M$ is a bounded metric space. In this sense, the space $Lip_0^{ph}(X)$ offers a slight structural advantage.
We now discuss the structure of the article and the key findings as follows:
\begin{enumerate}
 \item For any Banach space $X$ with dimension greater than $1$, we identify a proper closed infinite-dimensional subspace, namely $Lip_0^{ph}(X)$ (Space of positive homogeneous Lipschitz maps) of 
    $Lip_0(X)$ that contains $X^*$ 
    as a proper closed subspace, and we further show that 
    $\bigslant{Lip_0(X)}{Lip_0^{ph}(X)}$ is a dual space (see corollary \ref{cor lip q lipph}).
 \item It is well known that $Lip_0(\mathbb{R}, \mathbb{R})$ is linear isometrically isomorphic to $L^{\infty}(\mathbb{R})$ (see \cite{IROLFSOCDIDS}). In contrast, we see that the space $Lip_0^{ph}(\mathbb{R}, \mathbb{R})$ is two dimensional. However, once we move to $\mathbb{R}^2$, the space $Lip_0^{ph}(\mathbb{R}^2, \mathbb{R})$ becomes infinite-dimensional, as its predual is linearly isomorphic to $L^1(\mathbb{R})$ (see Example \ref{example}$(1)$). Thus, the imposition of positive homogeneity leads to a remarkable change in the dimensional structure of the corresponding Lipschitz spaces over $\mathbb{R}$. In fact, geometrically, the positively homogeneous Lipschitz maps from 
$\mathbb{R}$ to $\mathbb{R}$ are precisely the piecewise affine maps that may fail to be differentiable at most at a single point, namely the origin.
\item Analogous to the Lipschitz free space, we identify a canonical predual $F^{ph}(X)$ of $Lip_0^{ph}(X)$ and examine its relationship with the Lipschitz free space over $X$ in Proposition \ref{fph-->f}. Furthermore, we show that whenever $\dim(X) > 1$, the space $F^{ph}(X)$ contains a subspace that is linearly isomorphic to $L^1(\mathbb{R})$ (see Proposition \ref{fph cont l1r}). To prove this, we use the McShane-type extension theorem for positive homogeneous Lipschitz maps, proved in Section \ref{sec 4}.
    
\end{enumerate}
To move forward, we recall the following results from the theory of Lipschitz functions, which will be used in the subsequent sections. The canonical predual of $Lip_0(M)$; the Lipschitz free space associated with $M$ is denoted by $F(M)$. It is given by
$$
F(M) = \text{span} \{ \delta_x : x \in M \} \subset Lip_0(M)^*,
$$
where $\delta_x : Lip_0(M) \to \mathbb{R}$ is the linear functional defined by $\delta_x(f) = f(x)$ for all $f \in Lip_0(M)$. A key property of Lipschitz free spaces is their universal linearization of Lipschitz functions. Specifically, for any Lipschitz map $f : M \to N$ between pointed metric spaces such that $f(0_M) = 0_N$, there exists a bounded linear  operator $\Tilde{f} : F(M) \to F(N)$ satisfying $\|\Tilde{f}\| = Lip(f)$, and such that $\Tilde{f} \circ \delta_M = \delta_N \circ f$, where the map $\delta_M : M \to F(M)$ ($x \mapsto \delta_x$) is a nonlinear isometric embedding of a metric space into its Lipschitz free space. Another important property is the universal property of Lipschitz maps.  
Given $f \in Lip_0(M, X)$, there exists a unique bounded linear map $\hat{f} \in L(F(M), X)$ such that 
$$
\hat{f} \circ \delta_M = f
\quad \text{and} \quad
Lip(f) = \|\hat{f}\|.
$$
The nonlinear isometric embedding $\delta_X$ has a linear left inverse $\beta_X : F(X) \to X$, that acts on the basis elements as $\beta_X(\delta_x)=x$. For further details on these concepts, we refer to the important paper by Godefroy and Kalton~\cite{LFBS}, and the monograph by Weaver~\cite{LA2}, which provide a comprehensive study of Lipschitz functions. 

Throughout the text, we shall adapt the following notation. Given $X$ and $Y$ Banach spaces, $X \simeq Y$ stands for $X$ is linearly isomorphic to $Y$. $X\simeq^h Y$ stands for $X$ is linearly homeomorphic to $Y$. $X \cong Y$ stands for $X$ is linear isometrically isomorphic to $Y$. $X \xhookrightarrow{c} Y$ stands for $X$ is linear isomorphic to a complemented subspace of $Y$. The closed unit ball and unit sphere of $X$ are denoted by $B_X$ and $S_X$, respectively.

\section{Quotient of Finite-Dimensional Spaces in the Lipschitz Dual} \label{sec 2}
Let $\mathcal{A}$ be any $n$-dimensional subspace of $Lip_0(X,\mathbb{R})$. Therefore, $\mathcal{A}= span\left\{f_1,f_2,...,f_n\right\}$ for some $n$ linearly independent elements $f_1,f_2,...,f_n \in Lip_0(X)$. The following result provides an explicit identification of the space $ \bigslant{Lip_0(X)}{\mathcal{A}}$.
\begin{proposition}
    Let $f_1,f_2,...,f_n \in Lip_0(X)$ be linearly independent. Then $\bigslant{Lip_0(X)}{span\left\{f_1,f_2,...,f_n \right\}}$ is linear isometrically isomorphic to $\left(\bigcap\limits_{i=1}^{n} \ker(\hat{f_i})\right)^{*}$.
\end{proposition}
\begin{proof}
    We consider the closed subspace $M= \bigcap\limits_{i=1}^{n} \ker(\hat{f_i})$ of $F(X)$. As $Lip_0(X)$ and $F(X)$ form a dual pair, it is enough to show that $M^{\perp} = span\left\{f_1,f_2,...,f_n \right\}$. Since $F(X)^*$ is linear isometrically isomorphic to $Lip_0(X)$
    $$M^{\perp}= \left\{g \in Lip_0(X): \gamma(g)=0 \ \forall \ \gamma \in \bigcap\limits_{i=1}^{n} \ker(\hat{f_i})\right\}.$$
    Therefore, $\gamma(g)=0 \ \forall \ \gamma \in \bigcap\limits_{i=1}^{n} \ker(\hat{f_i})$ implies $\bigcap\limits_{i=1}^{n} \ker(\hat{f_i}) \subset \ker(\hat{g})$. Then it follows from a standard technique that $T_g \in span\left\{\hat{f_1},\hat{f_2},...,\hat{f_n} \right\}$. That is. there exists $\alpha_i \in \mathbb{R}$ for $i=1,2,...,n$ such that 
    \begin{eqnarray*}
        &\hat{g}=&\sum\limits_{i=1}^{n} \alpha_i \hat{f_i}\\
        &\iff& \hat{g}(\delta_x)=\sum\limits_{i=1}^{n} \alpha_i \hat{f_i}(\delta_x) \ \ \forall x \in X.\\
        &\iff& g \in span\left\{f_1,f_2,...,f_n \right\}.
    \end{eqnarray*}
    Hence we get $M^{\perp} = span\left\{f_1,f_2,...,f_n \right\}$ and from \cite[Theorem 4.9]{FA} the proof follows.
\end{proof}
\begin{remark}
\begin{enumerate}
    \item Although the above result is proved for Banach spaces, it remains valid when $X$ is replaced by a pointed metric space, and the proof proceeds by the same method.
\item If we consider $X= \mathbb{R}$ and the one dimensional subspace $span\{Id_{\mathbb{R}}\}$, then it follows from the result that $\bigslant{Lip_0(\mathbb{R})}{span\{Id_{\mathbb{R}}\}} \cong \ker(\beta_{\mathbb{R}})^*$, which has been shown as a corollary in \cite[Remark 5.5]{KarnMandal2025}.
\item In general, for any finite-dimensional subspace $\mathcal{A}$ of $Lip_0(X,Y)$, it is not known to the author whether the quotient space $\bigslant{Lip_0(X,Y)}{\mathcal{A}}$ can be identified with a space of linear operators.
\end{enumerate}
    
\end{remark}
 \section{A Proper Closed Subspace of the Lipschitz Dual Properly Containing the Linear Dual}
One of our other goal in this article is to identify a proper closed subspace of $Lip_0(X)$ that contains $X^*$ as a closed subspace. It is well-known that positive homogeneous uniformly continuous maps between Banach spaces are known to be Lipschitz. In seeking a proper closed subspace of $Lip_0(X)$ containing $X^{*}$, the author focused on positively homogeneous Lipschitz maps.
In pursuit of this objective, for any Banach spaces $X$ and $Y$, we consider a subset $S \subset X$ which is is closed under positive scalar multiplication (hence $0 \in S$), equipped with the metric induced by the norm on $X$, that is, $d(x,w) := \|x - w\|$ for all $x, w \in S$. 
We consider 
$$
Lip_0^{ph}(S, Y) := \left\{ f \in Lip_0(S, Y) : f(\alpha x) = \alpha f(x) \ \forall \, \alpha \ge 0, \, x \in X \right\},
$$ 
with the $Lip(\cdot)$ norm. Since any sequence $(f_n)$ in $Lip_0^{ph}(S, Y)$ converges to $f \in Lip_0(S, Y)$ in $Lip(\cdot)$ norm; implies $f_n$ converges to point-wise to $f$ also. It follows that the space $Lip_0^{ph}(S, Y)$ is a closed subspace of $Lip_0(S, Y)$. In particular, $Lip_0^{ph}(X, Y)$ contains $L(X, Y)$ as a closed subspace. Moreover, the inclusions are strict as the map $x \to \|x\|y_0$ is in $Lip_0^{ph}(X, Y) \setminus L(X, Y)$, where as the map $x \to \|x-x_0\|y_0 -\|x_0\|y_0$ is in $Lip_0(X, Y) \setminus Lip_0^{ph}(X, Y)$ for some fixed $x_0 \in X\setminus\{0\}$ and $y_0 \in S_Y$. For $Y=\mathbb{R}$ we denote the space of positive homogeneous real-valued Lipschitz maps on $X$ by $Lip_0^{ph}(X)$. 

Moreover, it is known that on the closed unit ball $ B_{Lip_0(X)}$, the $weak^*$ topology coincides with the topology of pointwise convergence.
It is also straightforward to verify that $ B_{Lip_0^{ph}(X)} $ is closed in $B_{Lip_0(X)}$ with respect to this topology.
Since $B_{Lip_0(X)}$ is compact under the topology of pointwise convergence, its closed subset $B_{Lip_0^{ph}(X)}$ is likewise compact.
Hence, by the Ng--Dixmier theorem~\cite[Theorem~1]{OATOD}, it follows that $Lip_0^{ph}(X)$ is isometrically isomorphic to a dual space.
In analogy with the predual of $Lip_0(X)$, we construct the predual of $Lip_0^{ph}(X)$. To proceed in this direction, we first present the following facts. The proofs follow similar lines to the classical case, while the differing parts can be derived from the hints provided herein.

\begin{facts} \label{facts}
\begin{enumerate}
   \item Recall the pointwise evaluation functional $\delta_x \in Lip_0(X)^{\ast}$. 
By $\delta_x^{ph}$ we denote its restriction to $Lip_0^{ph}(X)$, that is,
$
\delta_x^{ph} := \delta_x\big|_{Lip_0^{ph}(X)}.
$
We then define
$$
F^{ph}(X) := \overline{\operatorname{span}}^{\|\cdot\|}\left\{\delta_x^{ph} : x \in X\right\},
$$
which is a closed subspace of $Lip_0^{ph}(X)^{\ast}$.

\item Analogous to the map $\delta_X : X \to F(X)$ (a nonlinear isometry), the map 
$
\delta_X^{ph} : X \to F^{ph}(X), \qquad x \mapsto \delta_x^{ph},
$
is also a nonlinear isometry. 
However, the proof of this isometry differs slightly from the usual one, and we therefore provide its details below.

Let $x, y \in X$. Then $\|\delta_x^{ph} - \delta_y^{ph}\| = \sup\limits_{f \in B_{Lip_0^{ph}(X)}}|f(x)-f(y)| \leq \|x-y\|$. For the equality, consider the functional (exists due to Hahn-Banach theorem) $g \in S_{X^*} \left(\subset B_{Lip_0^{ph}(X)}\right)$ such that $g(x-y)=\|x-y\|$. Then $\|\delta_x^{ph} - \delta_y^{ph}\| \geq |g(x)-g(y)| = \|x-y\|$. Hence, the proof follows.
\item Since $\delta_X^{ph}$ is the restriction of $\delta_X$ to $Lip_0^{ph}(X)$, it is positive homogeneous—unlike the map $\delta_X$ itself.
\item  For any Banach space $X$, define
    $$\langle \cdot, \cdot \rangle : Lip_{0}^{ph}(X) \times F^{ph}(X) \xrightarrow{} \mathbb{R}$$ given by $$\langle f,\mu \rangle:= \mu(f)$$ for all $f \in Lip_{0}^{ph}(X)$ and $\mu \in F^{ph}(X)$. Then $\langle \cdot, \cdot \rangle$ is bilinear and non-degenerate. Also, it can be easily deduced that 
    $$\sup\limits_{f \in B_{Lip_{0}^{ph}(X, \mathbb{R})}} |\mu(f)| = \|\mu\| \ \ \mbox{and} \ \ \sup\limits_{\mu \in B_{F^{ph}(X)}}|\mu(f)|=Lip(f).$$
\item $Lip_{0}^{ph}(X, Y)$ is linear isometrically isomorphic to $L\left(F^{ph}(X), Y\right)$ via the map $f \mapsto T_f$, where $f$ and $T_f$ satisfy the following commutative diagram.
$$\begin{tikzcd}
& F^{ph}(X) \arrow[dr, "T_f"] & \\
X \arrow[ur, "\delta_X^{ph}"] \arrow[rr, "f"'] & & Y
\end{tikzcd}$$
The proof follows similar to the classical argument of proving $Lip_0(X, Y)$ is linear isometrically isomorphic to $L(F(X), Y)$, relying on the fact that 
$\delta_X^{ph}$ is positive homogeneous.
\item Let $X$ and $Y$ be Banach spaces and suppose $f \in Lip_0^{ph}(X, Y)$.  
Then there exists a unique linear map $C_f : F^{ph}(X) \to F^{ph}(Y)$ such that  
$
C_f \circ \delta_X^{ph} = \delta_Y^{ph} \circ f,$
that is, the following diagram commutes:
 $$\begin{tikzcd}
& F^{ph}(X) \arrow[rr, "C_f"] & & F^{ph}(Y) \\
X \arrow[ur, "\delta_X^{ph}"] \arrow[rr, "f"'] & & Y \arrow[ur, "\delta_Y^{ph}"']
\end{tikzcd}.$$
The proof of it follows similar to \cite[Lemma 2.2]{LFBS}.

\end{enumerate}
\end{facts}
\subsection{$Lip_{0}^{ph}(X)$ as a quotient space in $Lip_{0}(X)$}

\begin{proposition}
    For any Banach space $X$, consider the closed subspace of $F(X)$
$$
\overline{\operatorname{span}}^{\|\cdot\|}\bigl\{ r\delta_x - \delta_{rx} : r \ge 0, \ x \in X \bigr\}.
$$
Then, the pre-annihilator of $Lip_{0}^{ph}(X)$ is given by
$\ ^{\perp}Lip_{0}^{ph}(X)= \overline{\operatorname{span}}^{\|\cdot\|}\bigl\{ r\delta_x - \delta_{rx} : r \ge 0, \ x \in X \bigr\}.
$
\end{proposition}
\begin{proof}
    Since every $f \in Lip_{0}^{ph}(X)$ is positive homogeneous and $\ ^{\perp}Lip_{0}^{ph}(X)$ is a closed subspace of $F(X)$, we have 
$\overline{\operatorname{span}}^{\|\cdot\|}\{ r\delta_x - \delta_{rx} : r \ge 0, \ x \in X \} \subset \ ^{\perp}Lip_{0}^{ph}(X)$.

Now, suppose $\gamma \in \ ^{\perp}Lip_{0}^{ph}(X) \setminus \overline{\operatorname{span}}^{\|\cdot\|}\{ r\delta_x - \delta_{rx} : r \ge 0, \ x \in X \}$.
Then, by the Hahn-Banach separation theorem (see \cite[Theorem 3.5]{FA}), there exists $g \in Lip_0(X)$ such that $\gamma(g)=1$ and $\xi(g)=0$ for all $\xi \in \overline{\operatorname{span}}^{\|\cdot\|}\{ r\delta_x - \delta_{rx} : r \ge 0, \ x \in X \}$.
In particular, $(r\delta_x - \delta_{rx})(g)=0$ for all $x \in X$ and $r \ge 0$.
Therefore, $g \in Lip_{0}^{ph}(X)$, which implies $\gamma(g)=0$, a contradiction.
Hence, the claim follows.
\end{proof}
\begin{remark}
   A similar kind of arguments provides that 
$$\overline{\operatorname{span}}^{\|\cdot\|}\{ r\delta_x - \delta_{rx} : r \ge 0, \ x \in X \}^{\perp}= Lip_{0}^{ph}(X).$$ 
\end{remark} 
\begin{corollary} \label{cor lip q lipph}
    From the above facts we have the following :
    $$\bigslant{Lip_0(X)}{Lip_0^{ph}(X)} = \overline{\operatorname{span}}^{\|\cdot\|}\{ r\delta_x - \delta_{rx} : r \ge 0, \ x \in X \}^{\ast}.$$
\end{corollary}
The next result provides a classification of all positive homogeneous real-valued Lipschitz maps on a Banach space by relating them to real-valued Lipschitz maps whose domains are suitably restricted. On $S_X \cup \{0\}$ we consider the restricted metric induced by the norm $\|\cdot\|$ on the Banach space $X$.
\begin{proposition} \label{lph-->lsx}
   Let $X$ and $Z$ be a Banach spaces. Then $\left(Lip_0^{ph}(X, Z), Lip(\cdot)\right)$ is linearly homeomorphic to $\left(Lip_0(S_X \cup \{0\}, Z), Lip(\cdot)\right)$. 
\end{proposition}
\begin{proof}
    We define $\Lambda : Lip_0^{ph}(X, Z) \to Lip_0(S_X \cup \{0\}, Z)$ given by $\Lambda(f)= f\big|_{S_X \cup \{0\}}$. Clearly $\Lambda$ is a linear contraction. Using the positive homogeneity of the functions we can prove $ker(\Lambda)=\{0\}$. 

    Let $g \in Lip_0(S_X \cup \{0\}, Z)$. Put 
    $$h(x) = \begin{cases}
		\|x\|g\left(\frac{x}{\|x\|}\right), &~ \mbox{if} ~ x \ne 0 \\ 
		0, &~ \mbox{if} x=0.
	\end{cases}$$
    Then $h$ is positive homogeneous. We show that $h$ is Lipschitz also.
    Suppose $x, y \in X\setminus\{0\}$. Then 
    \begin{eqnarray*}
        |h(x)-h(y)| &=& \left|\|x\|g\left(\frac{x}{\|x\|}\right)- \|x\|g\left(\frac{y}{\|y\|}\right) + (\|x\|-\|y\|)g\left(\frac{y}{\|y\|}\right)\right|\\
        &\leq& Lip(g)\|x-y\| + Lip(g)\|x\| \left\|\frac{x}{\|x\|}-\frac{y}{\|y\|}\right\|.
    \end{eqnarray*}
    Again, by a standard technique, it follows that $\|x\| \left\|\frac{x}{\|x\|}-\frac{y}{\|y\|}\right\| \leq 2\|x-y\|$. Therefore, $|h(x)-h(y)| \leq 3 Lip(g) \|x-y\|$. The case is considerably simpler when either $x$ or $y$ equals $0$. Furthermore, it is straightforward to see that $\Lambda(h)=g$. This completes the proof.
\end{proof}
\begin{remark}
    For $Z=\mathbb{R}$ with dimension of $X$ greater equal to $2$, it follows from \cite[Theorem 3.2]{OTSOLFS} and the above proposition that $\ell_{\infty} \xhookrightarrow{} Lip_0^{ph}(X)$.
\end{remark}
This linear homeomorphism naturally induces a corresponding linear homeomorphism between their predual spaces, which we shall establish in the following proposition. Since the proof follows a similar way using this map $\Lambda$, we only provide the outline of it.
\begin{proposition} \label{fph-->f}
    For any Banach space $X$, $F(S_X)$ and $F^{ph}(X)$ are linearly homeomorphic.
\end{proposition}
\begin{proof}
    Let us define the map $\Theta: span\{\delta_x : x \in S_X\} \to F^{ph}(X)$ given by $\Theta \left(\sum\limits_{i=1}^{n} \alpha_i \delta_{x_i}\right)= \sum\limits_{i=1}^{n} \alpha_i \delta_{x_i}^{ph}$. Then the linearity of $\Theta$ follows from the definition and boundedness (contraction) holds using the map $\Lambda$ from the above proposition. Since $\Theta$ is continuous and $F^{ph}(X)$ is complete $\Theta$ can be extended continuously to  $F(S_X)$. 

    Again $\Phi: span\{\delta_x^{ph} : x \in S_X \} \to F(S_X)$ given by $\Phi \left(\sum\limits_{i=1}^{n} \alpha_i \delta_{y_i}^{ph}\right)= \sum\limits_{i=1}^{n} \alpha_i \delta_{y_i}^{ph}$ and sends $0$ to $0$. Then $\Phi$ is linear bounded (by the constant $3$), can be extended to $F^{ph}(X)$. Also, immediately it follows that $\Phi$ is the inverse of $\Theta$. This completes the proof.
\end{proof}
We work out a couple of examples for a better understanding of the space. 
\begin{example}
    Consider $X=\mathbb{R}$. Then $F(S_{\mathbb{R}})= F(\{\pm 1\})$ and $F^{ph}(\mathbb{R})$ are linearly homeomorphic. Therefore, $F^{ph}(\mathbb{R})$ is linearly homeomorphic to $\mathbb{R}^2$. Hence,  $Lip_0^{ph}(\mathbb{R}, \mathbb{R})$ is also linearly homeomorphic to $\mathbb{R}^2$. However, independently, we show that $Lip_0^{ph}(\mathbb{R}, \mathbb{R})$ is linear isometrically isomorphic to $(\mathbb{R}^2, \|\cdot\|_{\infty})$.

    For any $f \in Lip_0^{ph}(\mathbb{R}, \mathbb{R})$ we have 
    $$f(x) = \begin{cases}
		xf(1), &~ \mbox{if} ~ x \geq 0 \\ 
		-xf(-1), &~ \mbox{if}~ x<0.
	\end{cases}$$
    We define the following surjective linear map $\xi: Lip_0^{ph}(\mathbb{R}, \mathbb{R}) \to (\mathbb{R}^2, \|\cdot\|_{\infty})$ given by $\xi(f)=(f(1),f(-1))$. Clearly, $\sup\limits_{x,y \in [0, \infty), x \ne y} \frac{|f(x)-f(y)|}{|x-y|} = |f(1)|$ and $\sup\limits_{x,y \in ( -\infty, 0], x \ne y} \frac{|f(x)-f(y)|}{|x-y|} = |f(-1)|$.

    Suppose $x >0$ and $y=-t$ for some $t>0$. Then $f(x)-f(y)= xf(1)-tf(-1)$, and $|x-y|=x+t >0$. Now 
    \begin{eqnarray*}
        \sup\limits_{x>0,y<0} \frac{|f(x)-f(y)|}{|x-y|} &=& \sup\limits_{x,t>0} \frac{|xf(1)-tf(-1)|}{x+t}\\
        &=& \sup\limits_{x,t>0} \left|\frac{x}{x+t} f(1) + \frac{t}{x+t} (-f(-1))\right|\\
        &=& max\{|f(1
        )|,|f(-1)|\}
    \end{eqnarray*}
    Combining all the cases, we have $Lip(f)=\|(f(1),f(-1))\|_{\infty}$ and hence the isometry follows. 
\end{example}
In addition, it can be easily deduced that $Lip_0^{ph}([0,\infty), \mathbb{R})$ is one dimensional.
\begin{example} \label{example}
\begin{enumerate}
    \item From the proposition \ref{fph-->f} it follows that $F(S_{\mathbb{R}^2}) \simeq^h F^{ph}(\mathbb{R}^2)$. Moreover, from \cite{LFSITTISAGA} and \cite{ASOLFBS} we know that $F(S_{\mathbb{R}^2}) \simeq F(\mathbb{R}) \cong L^{1}(\mathbb{R})$. Therfore, we get $F^{ph}(\mathbb{R}^2)$ is linearly isomorphic to $L^1(\mathbb{R})$ and hence $Lip_0^{ph}(\mathbb{R}^2, \mathbb{R})$ is infinite dimensional.
\item In fact, later on we prove that for any normed linear space $X$ with $dim(X) \geq 2$($'dim'$ stands for dimension), $Lip_0^{ph}(X)$ is of infinite dimension. 
\end{enumerate}
\end{example}
\subsection{$Lip_0^{ph}(X, \mathbb{R})$ as a Banach algebra}
Here, we study the Banach space $Lip_0^{ph}(X, \mathbb{R})$ as an algebra over an arbitrary Banach space $X$. When the pointed metric space $(M, d, 0)$ is bounded, the space $\mathrm{Lip}_0(M)$ constitutes a Banach algebra in a weak sense; that is, there exists a constant $C > 1$ such that $Lip(fg) \leq C Lip(f) Lip(g)$ for all $f, g \in Lip_0(M)$. In fact, one may choose $C = 2\max\limits_{x \in M} d(x, 0)$. Such structures are referred to as \emph{Gelfand algebras} in \cite{LA2}. It is worth noting that in the usual definition of a Banach algebra, the above inequality is required to hold with $C = 1$. As discussed in \cite[Chapter~7]{LA2}, when $M$ is bounded, one can introduce an equivalent submultiplicative norm on $Lip_0(M)$, making it a genuine Banach algebra.
More recently, Fernando Albiac \textit{et al.} in \cite{LAALFSOUMS} introduced, for arbitrary metric spaces $M$, an explicit modification of the pointwise multiplication on $Lip_0(M)$ defined by
$f \odot g(x) := \dfrac{f(x)g(x)}{3d(x,0)}, \quad x \in M \setminus \{0\},$
for $f, g \in Lip_0(M)$, which turns $Lip_0(M)$, equipped with its standard norm, into a Banach algebra. In particular, $(Lip_0(X),Lip(\cdot))$ becomes a Banach algebra under this modified pointwise multiplication. Consequently, since $Lip_0^{ph}(X)$ is a closed subspace of $Lip_0(X)$, it forms a Banach subalgebra. Moreover, we introduce another modification—obtained by multiplying the usual pointwise product by a factor—defined via the map $\Lambda$ as in Proposition~\ref{lph-->lsx}, with respect to which $(Lip_0^{ph}(X), Lip(\cdot))$ itself becomes a Banach algebra.

Since $S_X \cup \{0\}$ is bounded the space $Lip_0(S_X \cup \{0\}, \mathbb{R})$ is closed under usual pointwise multiplication.
Hence, for any $f,g \in Lip_0^{ph}(X)$, it follows that $\Lambda(f)\Lambda(g) \in Lip_0(S_X \cup \{0\}, \mathbb{R})$. We now define 
\begin{eqnarray*}
    f\odot g(x)&:=& \frac{1}{5} \Lambda^{-1}\left(\Lambda(f)\Lambda(g)\right)(x)\\
     &=& \begin{cases}
		\frac{1}{5}\|x\|f\left(\frac{x}{\|x\|}\right)g\left(\frac{x}{\|x\|}\right), &~ \mbox{if} ~ x \neq 0 \\ 
		0, &~ \mbox{if}~ x=0.
	\end{cases}
\end{eqnarray*}
Directly, from the definition $f \odot g \in \mathrm{Lip}_0^{ph}(X)$.
 It only remains to show that $Lip(f \odot g) \leq Lip(f)Lip(g)$. Suppose $x,y \in X \setminus\{0\}$. Now
 \begin{eqnarray*}
     \left|\|x\|f\left(\frac{x}{\|x\|}\right)g\left(\frac{x}{\|x\|}\right)-\|y\|f\left(\frac{y}{\|y\|}\right)g\left(\frac{y}{\|y\|}\right)\right| &\leq& \left|\|x\|\left(f\left(\frac{x}{\|x\|}\right)-f\left(\frac{y}{\|y\|}\right)\right)g\left(\frac{x}{\|x\|}\right)\right| \\
     &+& \left|\|x\|f\left(\frac{y}{\|y\|}\right)\left(g\left(\frac{x}{\|x\|}\right) - g\left(\frac{y}{\|y\|}\right)\right)\right|\\
     &+& \left|(\|x\|-\|y\|)f\left(\frac{y}{\|y\|}\right)g\left(\frac{y}{\|y\|}\right)\right|\\
     &\leq& Lip(f)Lip(g)\|x\|\left|\frac{x}{\|x\|}-\frac{y}{\|y\|}\right| \\
     &+& Lip(f)Lip(g)\|x\|\left|\frac{x}{\|x\|}-\frac{y}{\|y\|}\right| + Lip(f)Lip(g)\|x-y\|\\
     &\leq& 5 Lip(f)Lip(g)\|x-y\|.
 \end{eqnarray*}
Further, for $x \neq 0$ and $y=0$ we get $|f\odot g(x)| \leq \frac{1}{5}Lip(f)Lip(g)\|x\|$. Hence, we have $Lip(f \odot g) \leq Lip(f)Lip(g)$.
\subsection{$F^{ph}(X)$ as a complemented copy in $F(X)\oplus \mathbb{R}$}
We record at this point that Benyamini and Sternfeld proved in \cite{SIIDNSALC} that there exists a Lipschitz retraction, say $R$, from $B_X$ onto $S_X$ whenever $X$ is an infinite dimensional normed linear space. Now we consider the following Lipschitz projection $G: X \to B_X$ given by 
$G(x)= \begin{cases}
		\frac{x}{\|x\|}, &~ \mbox{if} ~ \|x\| \geq 1 \\ 
		x, &~ \mbox{if}~ \|x\| \leq 1.
	\end{cases}$ Then $R \circ G: X \to S_X$ is a Lipschitz contraction. Note that $R \circ G(0)=R(0) \in S_X$ and hence $R \circ R(0)=R(0)$.
\begin{lemma} \label{lemma comp}
  $\overline{span}^{\|\cdot\|}\left\{\delta_x -\delta_{R(0)}: x \in S_X\right\}$ is linearly complemented in $F(X)$.  
\end{lemma} 
\begin{proof}
    Let us define the map $\xi: X \to \overline{span}^{\|\cdot\|}\left\{\delta_x -\delta_{R(0)}: x \in S_X\right\}$ given by $$\xi(x)= \delta_{R \circ G(x)} - \delta_{R(0)}.$$
    Then it follows immediately that $\xi \in Lip_0\left(X, \overline{span}^{\|\cdot\|}\left\{\delta_x -\delta_{R(0)}: x \in S_X\right\}\right)$. Then the linearization $\hat{\xi}$ of $\xi$ will work as the bounded linear projection from $F(X)$ onto $\overline{span}^{\|\cdot\|}\left\{\delta_x -\delta_{R(0)}: x \in S_X\right\}$, as for any $x \in S_X$, $\hat{\xi}(\delta_x - \delta_{R(0)})= \xi(x) - \xi(R(0))= \delta_x - \delta_{R(0)}$. 

Thus we have $\overline{span}^{\|\cdot\|}\left\{\delta_x -\delta_{R(0)}: x \in S_X\right\} \xhookrightarrow{c} F(X)$.
\end{proof}
\begin{lemma} \label{lemma codim}
    $\overline{span}^{\|\cdot\|}\left\{\delta_x -\delta_{R(0)}: x \in S_X\right\}$ is a closed co-dimension one subspace of $F(S_X)$.
\end{lemma}
\begin{proof}
    Let us define the map $Q : span\left\{\delta_x : x \in S_X\right\} \to \mathbb{R}$ given by $$Q\left(\sum\limits_{i=1}^{n}\alpha_i \delta_{x_i}\right)=\sum\limits_{i=1}^{n} \alpha_i.$$
    Clearly $Q$ is linear. For the boundedness of $Q$, consider $g(x)=min\{1, \|x\|\}$ for all $x \in X$. Then $g \in B_{Lip_0(X)}$ and
    $$\left\|\sum\limits_{i=1}^{n}\alpha_i \delta_{x_i}\right\| = \sup\limits_{f \in B_{Lip_0(X)}} \left|\sum\limits_{i=1}^{n}\alpha_i f(x_i)\right| \geq \left|\sum\limits_{i=1}^{n}\alpha_i g(x_i)\right| = \left|\sum\limits_{i=1}^{n}\alpha_i \right|.$$
    Since $Q$ is a contraction, it can be extended linearly to $F(S_X)=\overline{span}^{\|\cdot\|}\left\{\delta_x: x \in S_X\right\}$. Moreover, $Q(\delta_x -\delta_{R(0)})=0$ for any $x \in S_X$ implies $\overline{span}^{\|\cdot\|}\left\{\delta_x -\delta_{R(0)}: x \in S_X\right\} \subseteq \ker(Q)$. Again for $\mu \in \ker(Q) \cap span\left\{\delta_x : x \in S_X\right\}$ we have $\mu = \sum\limits_{i=1}^{n}\alpha_i \delta_{x_i}$ with $\sum\limits_{i=1}^{n}\alpha_i=0$. Thus $\mu = \sum\limits_{i=1}^{n}\alpha_i \left(\delta_{x_i}-\delta_{T(0)}\right)$. Therefore, by density $\ker(Q)=\overline{span}^{\|\cdot\|}\left\{\delta_x -\delta_{R(0)}: x \in S_X\right\}$. This completes the proof and in fact, we get $F(S_X) = span\left\{\delta_{R(0)}\right\} \oplus
    \overline{span}^{\|\cdot\|}\left\{\delta_x -\delta_{R(0)}: x \in S_X\right\}$
\end{proof}
\begin{corollary}
    From Proposition~\ref{fph-->f}, Lemma~\ref{lemma comp}, and Lemma~\ref{lemma codim}, it follows that for any infinite dimensional Banach space $X$,$F^{ph}(X)$ has an isomorphic copy contained in $F(X) \oplus \mathbb{R}$ as a subspace. 
\end{corollary}
\section{McShane’s type extension theorem for positive homogeneous real-valued Lipschitz functions} \label{sec 4}
 E.~J.~McShane, in \cite[Theorem~1]{EOROF}, proved that any real-valued function $f$ defined on a subset $E$ of a metric space $M$ can be extended to the entire space $M$ while preserving the same Lipschitz constant. In the next proposition, we establish an analogous extension theorem for a positive homogeneous Lipschitz map. The proof is obtained by making a slight modification to the classical one. For further details on McShane’s extension theorem, we also refer the reader to \cite[Theorem~4.1.1]{LF}. The necessary modification is presented in the following proposition, while the remaining arguments follow from the cited proof.

\begin{proposition} \label{ext for ph}
    Let $S$ be a subset of the Banach space $X$, which is closed under nonnegative scalar multiplication, and $f \in Lip_0^{ph}(S)$. Then $f$ can be extended to $X$ preserving the same Lipschitz constant.
\end{proposition}
\begin{proof}
    Let us define the functions $F, G : X \to \mathbb{R}$, defined for $x \in X$ by 
$$F(x) := \sup_{y \in S} \left(f(y) - Lip(f) \|x-y\|\right)$$ and 
$$G(x) := \inf_{y \in S} \left(f(y) + Lip(f) \|x-y\|\right).$$
We show that $F, G$ are positive homogeneous Lipschitz extensions of $f$ preserving the same Lipschitz constant $Lip(f)$, and any other positive homogeneous Lipschitz extension $H$ of $f$ with $Lip(H)=Lip(f)$ satisfies the inequalities $F \le H \le G$. We only show that $F$ and $G$ are positive homogeneous, and the rest of the proof follows the same as \cite[Theorem~4.1.1]{LF}. Since $0 \in S$, $F(0)=G(0)=0$. Let $t > 0$ and $x \in X$. 
Because $S$ is closed under positive scalar multiplication, the map $y \mapsto z := y/t$ is a bijection of $S$ onto itself.  
Therefore,
\begin{eqnarray*}
F(tx)
&=& \sup_{y \in S} \left( f(y) - Lip(f)\|tx - y\| \right) \\
&=& \sup_{z \in S} \left( f(tz) - Lip(f)\|tx - tz\| \right) \\
&=& \sup_{z \in S} \left( t f(z) - Lip(f) t \|x - z\| \right) \\
&=& t \sup_{z \in S} \left( f(z) - Lip(f)\|x - z\| \right) \\
&=& t F(x).
\end{eqnarray*}
Hence, $F \in Lip_0^{ph}(X)$ and similarly we can prove the same for $G$.
\end{proof}
We use this extension result to prove that $Lip_0^{ph}(X)$ is infinite dimensional for $dim(X)>1$. 
\begin{proposition} \label{fph cont l1r}
    Let $X$ be a Banach space with $dim(X)\geq 2$. Then $F^{ph}(X)$ contains a copy of $L^1(\mathbb{R})$ as a subspace, and therefore $Lip_0^{ph}(X)$ is infinite dimensional. 
\end{proposition}
\begin{proof}
    Let $x_1, x_2$ be two linearly independent vectors in $X$. We consider the norm on $\mathbb{R}^2$ induced by the norm on $X$, that is 
    $\|(a,b)\|_i = \|ax_1+bx_2 \|$ for all $(a,b) \in \mathbb{R}^2$. Then it follows that $\|\cdot\|_i$ is a norm on $\mathbb{R}^2$ and hence the mapping $\mathfrak{i} : \left(\mathbb{R}^2, \|\cdot\|_i\right) \to X$ given by $\mathfrak{i}(a,b)=a x_1 + b x_2$ is a linear isometry. 
    Put $Z= \mathfrak{i}(\mathbb{R}^2)$. then the inclusion map $i: Z \to X$ is the canonical embedding. Now recall that the classical linearization map of $\delta_X \circ i$ in the Lipschitz, namely $\widehat{\delta_X \circ i} : F(Z) \to F(X)$ becomes an isometric embedding. The proof relies on the fact that any $h \in Lip_0(Z)$ can be extended to $\hat{h} \in Lip_0(X)$ without altering the Lipschitz norm. Similarly, from Proposition~\ref{ext for ph}, any $h \in Lip_0^{ph}(Z)$ can be extended to $\hat{h} \in Lip_0^{ph}(X)$ without changing the Lipschitz norm. Hence, in this setup, $F^{ph}(\mathbb{R}^2)$ appears as a closed subspace of $F^{ph}(X)$. Therefore, by Example~\ref{example}, $F^{ph}(X)$ contains a copy of $L^1(\mathbb{R})$ as a subspace, and consequently, $Lip_0^{ph}(X)$ is infinite dimensional.
\end{proof}
{\bf Declaration.} 
The author was financially supported by the Senior Research Fellowship from the National Institute of Science Education and Research, Bhubaneswar, funded by the Department of Atomic Energy, Government of India. The author has no competing interests or conflicts of interest to declare relevant to this article's content.

\end{document}